\documentclass[reqno,12pt]{amsart}
\usepackage{DocumentPreamble}

\AddTitle{Exact cospectrality probabilities for uniform random matrices}
\AddAuthor{Alexander Van Werde}
\email{\href{mailto:a.van.werde@uni-muenster.de}{a.van.werde@uni-muenster.de}}
\address{University of Münster, Germany}

\AddMSC{15B52} 
\AddMSC{60B20} 
\AddMSC{15A21} 
\AddMSC{15B36}   
\AddMSC{11C20}  
\AddMSC{11D45}  

\AddKeyword{Cospectral}
\AddKeyword{rational orthogonal matrix}
\AddKeyword{algebraic number field}
\AddKeyword{random integral matrix}
\AddKeyword{Smith normal form}
\AddKeyword{Diophantine counting}

\begin{document}
\begin{abstract}  
    We study the conjugation action of orthogonal matrices on symmetric random matrices.
    Given a fixed orthogonal matrix over an algebraic number field and a random matrix with entries sufficiently uniform in the ring of integers, we wonder what the probability is that the conjugate is again integral.  
    Our main result establishes an exact formula for this probability in terms of the Smith ideals associated to the orthogonal matrix.  
    
    As an illustrative application, we establish exact formulas for the expected number of rational orthogonal matrices that preserve the integrality of a random matrix for every fixed denominator in dimensions two and three.  
    Notably, the dependence on the denominator turns out to be non-monotone due to number-theoretic fluctuations.
    We also prove bounds on the probability of rational cospectrality with bounded but arbitrarily large denominator. 
\end{abstract}
\maketitle
\section{Introduction}
Kac famously popularized the question ``Can one hear the shape of a drum?'' that asks whether eigenvalues of the Laplacian operator characterize a domain uniquely \cite{kac1966can}.  
While there exist examples where the answer is negative \cite{
gordon1992one,milnor1964eigenvalues}, results by Wolpert imply that the answer is positive in the generic case \cite{wolpert1979length}.

A similar question can be posed for discrete matrices. 
G\"unthard and Primas asked in 1956 whether the adjacency spectrum of a graph characterizes it up to isomorphism \cite{gunthard1956zusammenhang}.
Here too, negative examples can be found \cite{collatz1957}. 
A fundamental conjecture by Haemers \cite{haemers2016almost,van2003graphs} however proposes that the answer should be positive in the generic case: a random graph is believed to be characterized by its adjacency spectrum with high probability. 
A related \emph{fingerprint conjecture} has also been proposed more recently by Vu for random $\pm 1$ matrices \cite{VU2014combinatorial,vu2021recent}. 

The fact that we restrict ourselves to a subset of all matrices is essential for spectral characterization problems.
Indeed, given any symmetric matrix $\bX \in \bbR^{n\times n}$ that is not a multiple of the identity, there is a manifold of other symmetric matrices with the same spectrum given by $\bQ^{\T} \bX \bQ$ with $\bQ$ running over all orthogonal matrices.
The crux in the spectral characterization problem is whether this manifold has non-trivial intersection with the considered set of valid matrices. 
This raises the question: if a symmetric random matrix $\bX$ is constrained to have entries in some given discrete set and $\bQ$ is an orthogonal matrix, what is the probability that $\bQ^{\T}\bX \bQ$ has entries in this same discrete set? 

Significant progress in this direction has recently been made in an asymptotic setting for restricted classes of orthogonal matrices. 
For example, Wang and Zhao \cite{wang2025almost} considered the probability that $\bQ^{\T}\bX\bQ\in \bbZ^{n\times n}$ when $\bQ$ is a rational orthogonal matrix and $\bX$ is the adjacency matrix of a random graph.
They gave an upper bound for fixed $\bQ$ and combined this estimate with a union bound to establish that random graphs are not cospectral through rational matrices with bounded denominator with high probability; a special case of Haemers' conjecture. 
The methods in \cite{wang2025almost} also apply when $\bX$ is a random matrix with integer entries. 
This and other related works are discussed further in Section \ref{sec: RelatedWork}.

Our goal in the present paper is to determine exactly how the probability that $\bQ^{\T} \bX \bQ$ is a valid matrix depends on the orthogonal matrix $\bQ$, which we do in a non-asymptotic setting and not only for rational orthogonal matrices but also for cases with algebraic entries.
More precisely, we consider the scenario where $\bQ$ has entries in an algebraic number field and where $\bX$ is a symmetric matrix with entries in the field's ring of integers that are sufficiently uniformly distributed. 
Theorem \ref{thm: MainExact} provides an exact formula for the probability that the conjugate is again in the ring of integers.   
This result is already new when the algebraic number field is given by the rational numbers, in which case it concerns the probability that $\bQ^{\T} \bX \bQ\in \bbZ^{n\times n}$ for a symmetric random matrix $\bX \in \bbZ^{n\times n}$.

We give illustrative applications of the obtained formulas in Section \ref{sec: LocalRationalConjugation}. 
Instead of considering an $n\times n$ matrix with entries on $\{0,1 \}$ as $n$ tends to infinity as in \cite{wang2025almost}, we fix $n$ and take entries uniform on $\{0,1,\ldots,k \}$ with $k$ large. 
Theorems \ref{thm: Switching} and \ref{thm: 3Switching} study the expected number of rational orthogonal matrices with a given denominator whose conjugation preserves integrality for $n\leq 3$.
Notably, the dependence on the denominator turns out to be non-monotone due to fluctuations of a number-theoretic nature. 
This is a new phenomenon for the low-dimensional setting with integer matrices.

In particular, formulas for the expected value imply upper bounds on the probability of rational cospectrality.
As a consequence of our results with $n\leq 3$, we show in Corollary \ref{cor: Probability} that the probability of rational cospectrality with arbitrarily fixed denominator is bounded away from one, thus giving variants of \cite{wang2025almost} for low-dimensional matrices.
Here,  probability bounded away from one is the best that one can hope for when the dimensionality is fixed. 

\subsubsection*{Outline} 
We next state our main result in Section \ref{sec: ExactFormula}, consider applications to rational cospectrality in Section \ref{sec: LocalRationalConjugation}, and discuss related literature in Section \ref{sec: RelatedWork}. 
The proofs are given in Sections \ref{sec: ProofThmExact} and \ref{sec: ProofApplication}.

\pagebreak[3]
\subsection{Exact formula}\label{sec: ExactFormula}
We start by introducing notions that are required to state the result.
Consider an algebraic number field $K$ and let $\sO_K \subseteq K$ denote its ring of integers. 
Every algebraic number $k\in K$ can be written as $k = a/m$ with $m\in \bbZ$ and $a\in \sO_K$ \cite[\S2]{neukirch1999algebraic}. 
Thus, the following is well-defined: 

\begin{definition}\label{def: Level}
    The \emph{level} of a matrix $\bQ \in K^{n\times n}$ is the least positive integer $\ell\geq 1$ with the property that $\ell\bQ \in \sO_K^{n\times n}$.     
\end{definition}

Recall that the ring of integers is always a Dedekind domain \cite[\S3]{neukirch1999algebraic}. 
Consequently, for every ideals $\mathfrak{a},\mathfrak{b}\subseteq \sO_K$ with $\mathfrak{a} \subseteq \mathfrak{b}$ and $\mathfrak{b}\neq 0$ there exists a unique ideal $\mathfrak{c}\subseteq \sO_K$ with $\mathfrak{b}\mathfrak{c} = \mathfrak{a}$ \cite[p.41]{marcus1977number}. 
Hence, the following is well-defined: 

\begin{definition}\label{def: DeterminantalIdeal}
    The \emph{$i$th determinantal ideal} $\mathfrak{D}_i\subseteq\sO_K$ of a matrix $\bG\in \sO_{K}^{n\times n}$ is the ideal generated by all $i\times i$ minors.
    The \emph{Smith ideals} $\mathfrak{d}_1,\ldots,\mathfrak{d}_n\subseteq\sO_K$ are defined by $\mathfrak{d}_1 = \mathfrak{D}_1$ and $\mathfrak{d}_i \mathfrak{D}_{i-1}=\mathfrak{D}_i$ for every $i>1$ with $\mathfrak{D}_{i-1} \neq 0$ and $\mathfrak{d}_i \de 0$ otherwise.  
\end{definition}
\pagebreak[3]
Finally, let us recall that the \emph{norm} of an ideal $\mathfrak{a} \subseteq\sO_K$ is the cardinality of the quotient ring $\cN_{\sO_K}(\mathfrak{a})\de \#(\sO_K /\mathfrak{a})$. 
We have the following exact formula:
\begin{theorem}\label{thm: MainExact}
    Let $\bQ\in K^{n\times n}$ be a matrix with level $\ell$ that is orthogonal. 
    Consider a random symmetric $n\times n$ matrix $\bX\in \sO_K^{n\times n}$ whose upper-triangular entries $\{\bX_{i,j}:i\leq j \}$ have independent and uniformly distributed reductions in $\sO_K/\ell^2 \sO_K$. 
    Then, with $\mathfrak{d}_i$ the Smith ideals of $\ell \bQ$, 
    \begin{equation}
        \bbP\bigl(\bQ^{\T} \bX \bQ \in \sO_K^{n\times n} \bigr) = \prod_{i=1}^{\lfloor n/2 \rfloor} \prod_{j=i}^{n-i} \frac{\cN_{\sO_K}(\mathfrak{d}_i \mathfrak{d}_j)}{\cN_{\sO_K}(\ell^2 \sO_K)}.  \label{eq:GrumpyDen}
    \end{equation}
\end{theorem}

The proof is given in Section \ref{sec: ProofThmExact}. 
We establish a formula for general matrices over a Dedekind domain without orthogonality constraint in Theorem \ref{thm: GeneralQ}, and deduce \eqref{eq:GrumpyDen} by proving that orthogonality implies restrictions on the Smith ideals in Lemma \ref{lem: ProductSmith}.
The method of proof for Theorem \ref{thm: GeneralQ} relies on invariance properties of the random matrix $\bX$ to replace $\bQ$ by a diagonal matrix. 
This method can also be applied to other random matrix ensembles. 
For instance, it also applies to Hermitian random matrices or to matrices without symmetry constraint, so long as the entries are uniformly distributed; see Remark \ref{rem: Invariance}.

It is important to note that \eqref{eq:GrumpyDen} and the required assumption on $\bX$ not only depends on the rational matrix $\bQ$, but also on the algebraic number field $K$:

\begin{example}\label{ex: 2x2}
    Consider the $2\times 2$ orthogonal matrix $\bQ$ defined by 
    \begin{equation}
    \bQ \de \frac{1}{5}\begin{pmatrix}
        3 & 4\\ 
        4 & -3
    \end{pmatrix}\label{eq:OddLog}
    \end{equation}
    For $K = \bbQ$ we have that $\sO_K = \bbZ$ and the level of $\bQ$ is $5$. 
    Further, the determinantal ideals of $5\bQ$ are $\mathfrak{D}_1 = \bbZ$ and $\mathfrak{D}_2 = 25\bbZ$.
    The Smith ideals then equal the determinantal ideals. 
    Consequently, with $\operatorname{SymUnif}_{n\times n}(S)$ the uniform distribution on symmetric $n\times n$ matrices with entries in a finite set $S$, Theorem \ref{thm: MainExact} yields that 
    \begin{equation}
        \bbP\bigl(\bQ^{\T} \bX \bQ \in \bbZ^{n\times n} \bigr)  = 1/25\quad \textnormal{for}\quad \bX \sim \operatorname{SymUnif}_{2\times 2}\{0,1,\ldots,24 \}.  \label{eq:SpookyPaint}
    \end{equation}
    \pagebreak[3]
    We could also view $\bQ$ as a matrix over some other algebraic number field, which does not have to be real for Theorem \ref{thm: MainExact} to apply. 
    For instance, if $K = \bbQ(i)$ with $i \de \sqrt{-1}$ then $\sO_K = \bbZ[i]$ and it follows from $\cN_{\bbZ[i]}(25\bbZ[i]) = 25^2 = 625$ that 
    \begin{equation}
        \bbP\bigl(\bQ^{\T} \bX \bQ \in \bbZ[i]^{n\times n} \bigr) = 1/625  \ \  \textnormal{for}\ \ \bX \sim \operatorname{SymUnif}_{2\times 2}\bigl(a + bi : a,b\in \{0,1,\ldots,24 \} \bigr).\nonumber
    \end{equation}
    Of course, if $K$ is not real, then is may be more natural to consider random Hermitian matrices instead of symmetric matrices. 
    Results that can cover such cases also follow from our proofs; see Remark \ref{rem: Invariance}. 
\end{example}

\subsection{Application to rational cospectrality}\label{sec: LocalRationalConjugation}

Let us now consider $\bK = \bbQ$.
Then, the level of a rational orthogonal matrix $\bQ$ is the least integer $\ell\geq 1$ with $\ell \bQ \in \bbZ^{n\times n}$ and the Smith ideals are of the form $\mathfrak{d}_i = d_i \bbZ$ for $d_i \in \bbZ$ since $\bbZ$ is a principal ideal domain. 
These $d_i$ recover the Smith normal form of the integer matrix $\ell \bQ$.
\begin{corollary}\label{cor: RationalConjugation}
    Suppose that $\bX$ is a random symmetric integer $n\times n$ matrix whose upper-triangular entries $\{\bX_{i,j}: i\leq j \}$ have independent and uniformly distributed reductions to $\bbZ/\ell^2 \bbZ$. 
    Then, 
    $
        \bbP(\bQ^{\T} \bX \bQ \in \bbZ^{n\times n}) = \prod_{i=1}^{\lfloor n/2\rfloor} \prod_{j=i}^{n-i} (d_i d_j/\ell^2).
    $
\end{corollary} 
Corollary \ref{cor: RationalConjugation} can be used to investigate how the frequency of rational cospectrality depends on the level for fixed dimensionality $n$.
The additional required ingredient is an estimate on the number of rational orthogonal matrices with a given denominator and Smith normal form.

We next illustrate this application for $n= 2$ and $n=3$. 
It turns out that exact results can be attained in these cases, which allows us to showcase some new phenomena for the low-dimensional setting. 
The extension to arbitrary $n$ is an interesting open problem that we expect to be challenging. 
We suspect that exact results (with simple formulation) may not be possible for large $n$ and that upper bounds may be a more fruitful question.

Denote $\bbO_n(\ell , \bbQ)$ for the set of all rational orthogonal matrices $\bQ$ with level $\ell\geq 1$.
Then, we count the number of conjugacies of level $\ell$ for $\bX \in \bbZ^{n\times n}$ as 
\begin{equation}
    N_n(\ell) \de \#\bigl\{\bQ \in \bbO_n(\ell , \bbQ): \bQ^{\T} \bX \bQ \in \bbZ^{n\times n} \bigr\}. \label{eq:QuickPaint}
\end{equation}
We give exact formulas for the expected value for $n\leq 3$ and arbitrary $\ell \geq 2$. 
It suffices to consider levels greater than one because $\ell = 1$ corresponds to signed permutations of the rows and columns which trivially preserve integrality. 
\begin{theorem}\label{thm: Switching}
    Fix some arbitrary $\ell \geq 2$ and suppose that $\bX$ is a random symmetric $2\times 2$ matrix as in Corollary \ref{cor: RationalConjugation} with $n =  2$.  
    Then, $\bbO_2(\ell , \bbQ) \neq \emptyset$ if and only if the factorization of $\ell$ into powers of distinct primes is of the form 
    \begin{equation}
        \ell =  p_1^{k_1} \cdots p_r^{k_r} \   \textnormal{ with } \  p_i \equiv 1 \bmod 4\  \textnormal{ for every }i\leq r.\label{eq:KnownToy} 
    \end{equation}
    Moreover, if $\ell$ is as in  \eqref{eq:KnownToy} then $\#\bbO_2(\ell , \bbQ) = 2^{r+3}$ and consequently 
    \begin{equation}
        \bbE\bigl[N_2(\ell)\bigr] =    \frac{2^{r+3}}{\ell^2}.\label{eq:GhostlyFan}    \end{equation}
\end{theorem}

\begin{theorem}\label{thm: 3Switching}
    Fix some $\ell \geq 2$ and let $\bX$ be a random symmetric $3\times 3$ matrix as in Corollary \ref{cor: RationalConjugation} with $n =  3$.  
    Then, $\bbO_3(\ell , \bbQ) \neq \emptyset$ if and only $\ell$ is odd. 
    In that case, 
    \begin{equation}
        \bbE[N_3(\ell)] = \frac{48}{\ell^2} \prod_{p\mid \ell}\Bigl(1 + \frac{1}{p} \Bigr), \label{eq:KindRock}
    \end{equation}
    where the product runs over the distinct prime divisors of $\ell$.    
\end{theorem}
The proofs of Theorems \ref{thm: Switching} and \ref{thm: 3Switching} rely on reductions to known Diophantine counting problems. 
For instance, counting $2\times 2$ orthogonal matrices at level $\ell$ can be reduced to counting integers $a,b$ with $a^2 + b^2 = \ell^2$ and $\operatorname{gcd}(a,b)=1$, which is a classical problem \cite[Eq.(1.6)]{cooper2007number}. 
Using Corollary \ref{cor: RationalConjugation} and that a $2\times 2$ orthogonal matrix at level $\ell$ always has Smith normal form $(d_1,d_2) =(1,\ell^2)$ then yields \eqref{eq:GhostlyFan}. 
The proof of \eqref{eq:KindRock} is similar in spirit and relies on the Euler--Rodrigues parametrization of the special orthogonal group in dimension three.
Details are given in Sections \ref{sec: ProofThm2} and \ref{sec: ProofThm3}.

It is interesting to note that \eqref{eq:GhostlyFan} and \eqref{eq:KindRock} are not monotone in $\ell$, even constrained to those values where the expected value is nonzero.
Specifically, while the the factor $1/\ell^2$ ensures that the average trend is downwards as $\ell$ grows large, the prime factorization  can fluctuate from one number to the next.
This phenomenon is clearly visible in Figure \ref{fig:combined_plots} where we see that the factor $2^r$ gives rise to parallel streaks on a log-log plot for $n=2$, while the pattern is less orderly for $n=3$ because the factor $\prod_{p\mid \ell} (1+1/p)$ depends not only on the number of prime factors but also on their size.

\begin{figure}[h]
    \centering
    \begin{subfigure}[b]{0.495\textwidth}
        \centering
        \begin{tikzpicture}
            \begin{loglogaxis}[
                width=5.15cm, height=8cm,
                xlabel={Level $\ell$},
                ylabel={$\mathbb{E}[N_2(\ell)]$}, 
                label style={font=\scriptsize},      
                tick label style={font=\scriptsize},
                scale only axis,
                xmin=3.0473398664218925, xmax=164059.48201210098,
                ymin=5.9445270913812e-10, ymax=1.7229719058083621,
                axis on top
            ]
            \addplot graphics[
                xmin=3.0473398664218925, xmax=164059.48201210098,
                ymin=5.9445270913812e-10, ymax=1.7229719058083621
            ] {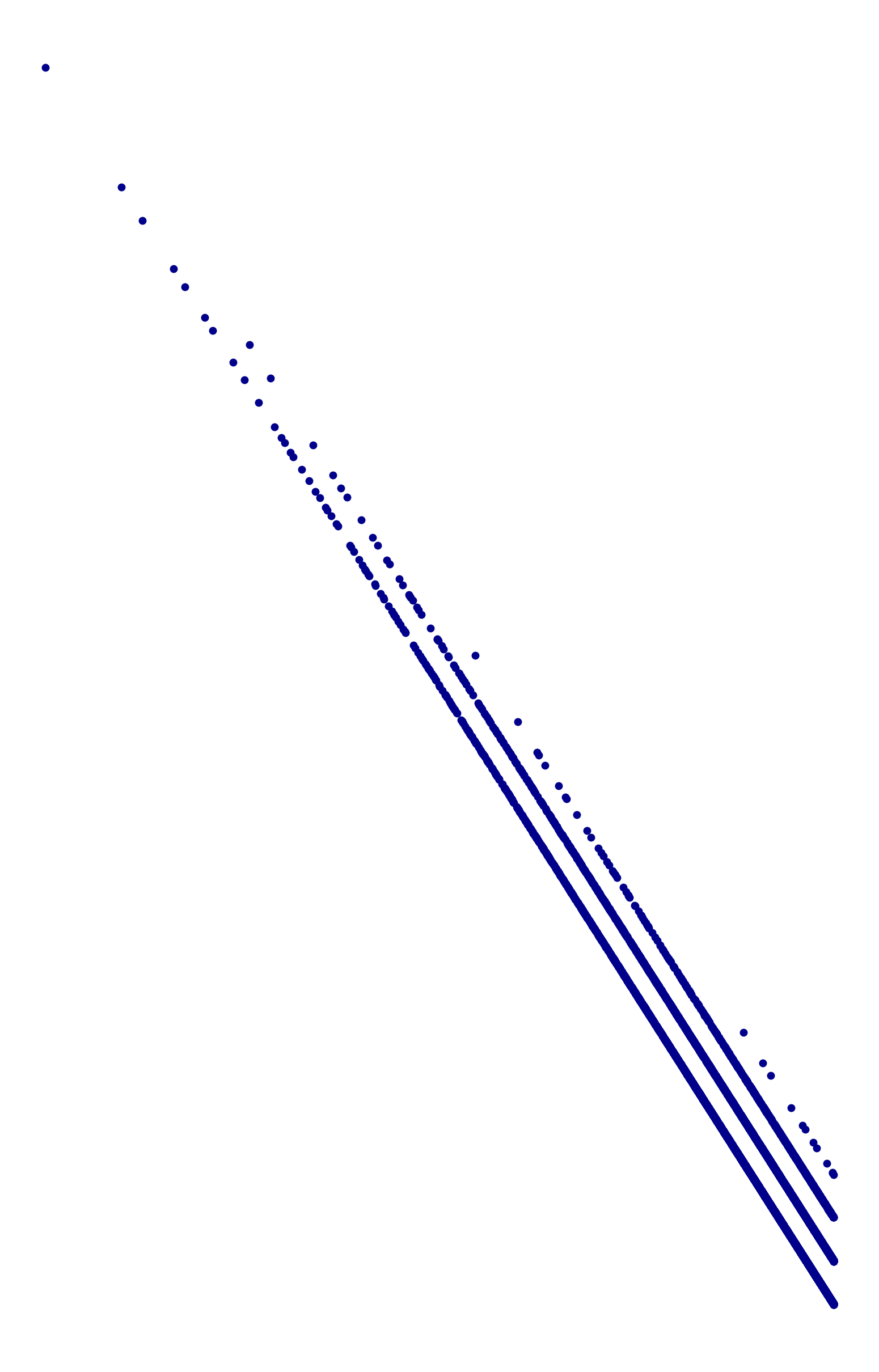};
            \end{loglogaxis}
        \end{tikzpicture}
    \end{subfigure}
    \hfill  
    \begin{subfigure}[b]{0.495\textwidth}
        \centering
        \begin{tikzpicture}
            \begin{loglogaxis}[
       width=5.15cm, height=8cm,
                xlabel={Level $\ell$},
                ylabel={$\mathbb{E}[N_3(\ell)]$}, 
                label style={font=\scriptsize},      
                tick label style={font=\scriptsize},
                scale only axis,
        xmin=1.7822865700776036, xmax=168321.41645265138,
        ymin=8.351472899805355e-10, ymax=10.219700164859102,
        axis on top
    ]
    \addplot graphics[
        xmin=1.7822865700776036, xmax=168321.41645265138,
        ymin=8.351472899805355e-10, ymax=10.219700164859102
    ] {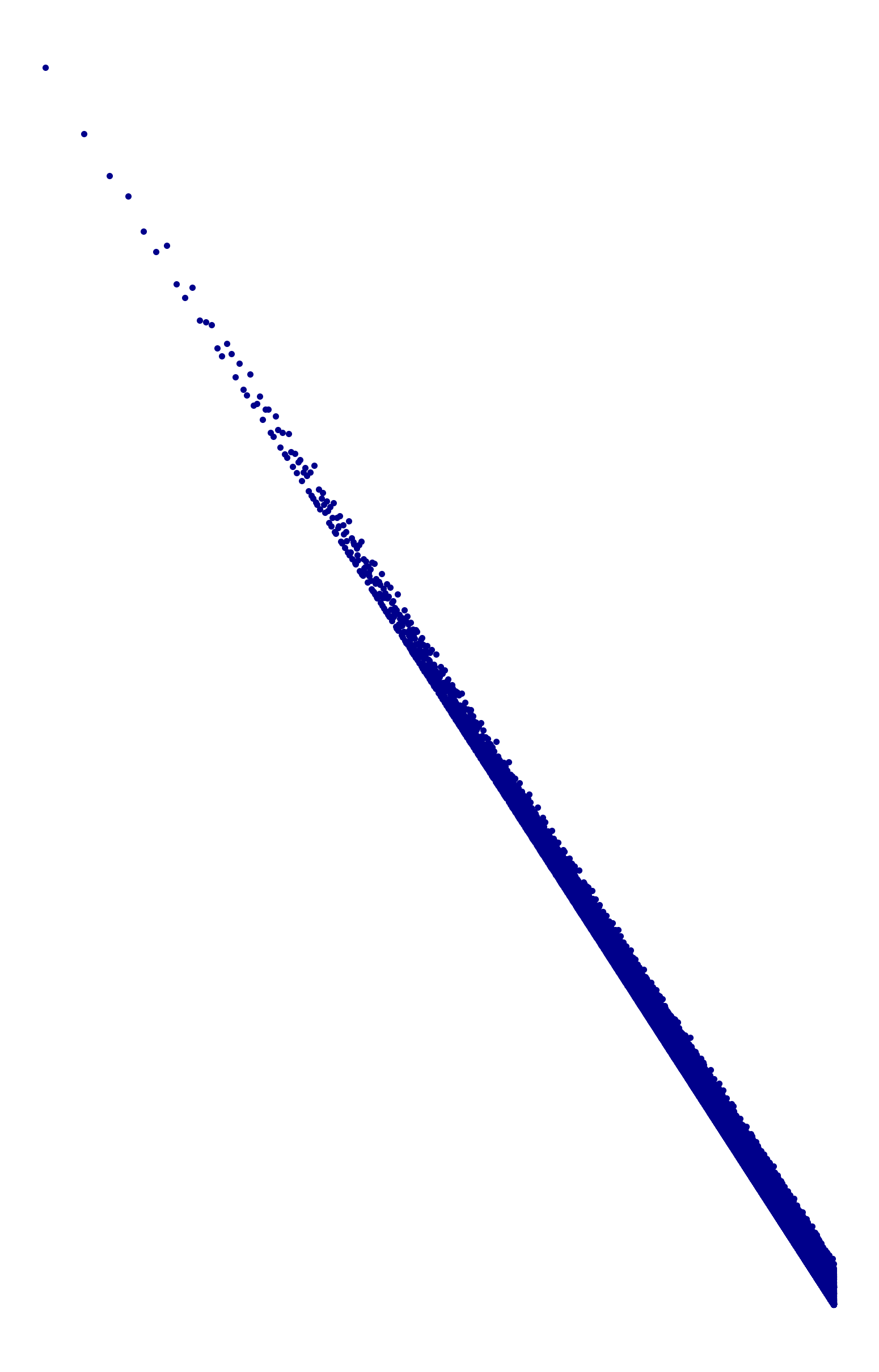};
    \end{loglogaxis}
    \end{tikzpicture}
    \end{subfigure}
    \caption{The sequence $\ell \mapsto \bbE[N_n(\ell)]$ is not monotone for $n=2,3$.}
    \label{fig:combined_plots}
\end{figure}

The expected value $\bbE[N_n(\ell)]$ can, in particular, be used to give upper bounds on the probability that $N_n(\ell) \neq 0$. 
This enables estimates on the probability of cospectrality at bounded but arbitrarily large level. 
The proof of the following result and explicit constants are given in Section \ref{sec: ProofCorProbability}. 
\begin{corollary}\label{cor: Probability}
    For any $k,n\geq 1$ let $\bX_k$ be a random symmetric $n\times n$ matrix whose upper-triangle has independent and uniform entries on $\{1,\ldots,k \}$. Then, for $n\leq 3$,
    \begin{equation}
        \lim_{L \to \infty} \lim_{k\to \infty} \bbP\Bigl( \exists \bQ \in \cup_{\ell=2}^L\bbO_n(\ell , \bbQ): \bQ^{\T} \bX_k \bQ \in \bbZ^{n\times n} \Bigr) < 1. \label{eq:IdleUrsa} 
    \end{equation}
\end{corollary}
\pagebreak[3]
It would be interesting future work if one could show that the order of the limits in \eqref{eq:IdleUrsa} may be exchanged, as this would imply the strictly stronger claim that there is a non-vanishing probability that $\bQ^{\T}\bX_k  \bQ \not\in \bbZ^{2\times 2}$ for \emph{all} rational orthogonal matrices $\bQ$ that are not signed permutations. 
The latter does not follow from the methods in the current paper since it also requires controlling levels $\ell\gg k$.  
Here, note that non-vanishing probability is the best that one can hope for when the dimensionality $n$ is fixed. 
We propose the following: 

\begin{conjecture}
    Fix some $n\geq 1$.
    Let $\bX_k$ be a random symmetric $n\times n$ matrix whose upper-triangular entries are independent and uniform on $\{1,\ldots,k \}$.  
    Then, 
    \begin{equation}
        \limsup_{k\to \infty} \bbP\bigl(\exists \bQ  \in \cup_{\ell=2}^\infty \bbO_n(\ell , \bbQ): \bQ^{\T}\bX_k \bQ \in \bbZ^{n\times n}  \bigr) <1.
    \end{equation}
\end{conjecture}
The natural variants of this conjecture for algebraic number fields would also be of considerable interest.

\subsection{Related work}\label{sec: RelatedWork}
The new finding in Corollary \ref{cor: RationalConjugation} that the Smith normal form determines the probability yields a natural connection to recent work on deterministic theory. 
In particular, Qiu, Wang, and Zhang  \cite{qiu2023smith} emphasized the Smith normal form of a graph's \emph{walk matrix} in deterministic sufficient conditions for characterization by generalized spectrum that improve upon earlier conditions by Wang \cite{wang2017simple}.\footnote{The generalized setting is equivalent to imposing that the cospectrality arises through an orthogonal matrix $\bQ$ with $\bQ e  = e$ with $e$ the all-ones vector; see also Johnson and Newman \cite{johnson1980note} for the origins of this notion as well as equivalent formulations.}
The inverse of the walk matrix there relates to the orthogonal matrix realizing the cospectrality, which has to be rational in their setting \cite[Lemma 2.1]{qiu2023smith}. 
In particular, the Smith normal forms are related. 
This shared feature seems to hint at deeper connections. 
For instance, it is natural to wonder if Smith ideals can be used in conditions for notions of cospectrality over algebraic number fields. 

\pagebreak[3]

Deterministic sufficient conditions for an integer symmetric matrix $\bX\in \bbZ^{n\times n}$ that imply that $\bQ^{\T}\bX \bQ\not\in \bbZ^{n\times n}$ for \emph{all} rational orthogonal matrices $\bQ$ that are not signed permutations have been developed by Wang and Yu \cite{wang2016square} as well as by Wang, Yang, and Zhang \cite{wang2024rational}. 
Those conditions allow ruling out arbitrary level. 
Numerical evidence suggest that they apply with non-vanishing frequency, but no rigorous bounds are presently known. 
Specific conjectures will appear in forthcoming work of Lvov and the present author \cite{vanwerde2026satisfaction}. 
One of the ingredients in our proofs is inspired by a related argument from that forthcoming work; see the discussion preceding Lemma \ref{lem: UniformSymQ}. 

Recall from the introduction that \cite{wang2025almost} proved that $\bQ^{\T}\bX \bQ \not\in \bbZ^{n\times n}$ for all $\bQ\in \bbO_n(\ell,\bbQ)$ for fixed $\ell$ with high probability as $n$ tends to infinity. 
Their results improved upon earlier work of Wang and Xu \cite{wang2010asymptotic} that concerned a special case with level two with the additional constraint that $\bQ e = e$ for $e=(1,\ldots,1)^{\T}$. 
The proof methods in \cite{wang2025almost,wang2010asymptotic} exploit that $\bQ$ has to be sparse when its level is small relative to the dimensionality.
Note that such sparsity is not needed in the present paper.
It is would be interesting future work if one could also use the Smith normal form in bounds for matrices with non-uniform entries, as this could enable extending \cite{wang2025almost} to larger levels or fixed $n$ when sparsity is not valid.

\pagebreak[3]
A setting without rationality constraint but instead assuming that the orthogonal matrix only has a small nontrivial block has recently been studied by Abiad, Van de Berg, and Simoens  \cite{abiad2025counting}. 
More precisely, they considered the \emph{switching method} associated to a fixed orthogonal matrix $\bQ_{m\times m} \in \bbR^{m\times m}$, which amounts to conjugation by $\bQ = \bP \operatorname{diag}(\bQ_{m\times m}, \b1_{n-m})$ with $\bP$ a permutation matrix and $\b1_{n-m}$ the identity.
The results in \cite{abiad2025counting} determine the asymptotic probability that a given switching method produces another graph that is moreover non-isomorphic to the original graph.  
Our main results could potentially be a starting point for generalizations of those results in the direction of random integral matrices.

Let us finally stress that the Smith and determinantal ideals in Theorem \ref{thm: MainExact} come from the orthogonal matrix $\bQ$, not from the matrix $\bX$. 
The different question of what can be deduced from ideals associated to $\bX$ has however also been considered by various authors. 
We refer the interested reader to \cite{elsheikh2015relating,newman1991matrices,rushanan1995eigenvalues} for works that relate properties of the Smith ideals to the eigenvalues of a deterministic integer matrix $\bX$, and to \cite{abiad2022codeterminantal,abiad2025distinguishing} for investigations concerning graphs that share the same determinantal ideals.

\section{Proof of \texorpdfstring{Theorem \ref{thm: MainExact}}{Theorem}}\label{sec: ProofThmExact}
Let $R$ be a Dedekind domain and fix an ideal $I\subseteq R$ with $R/I$ finite. 
It suffices to study the probability that $\bG^{\T} \bX \bG \equiv 0 \bmod I$ for a fixed matrix $\bG \in R^{n\times n}$.  
Indeed, Theorem \ref{thm: MainExact} then follows by taking $R = \sO_K$, $I = \ell^2 \sO_K$ and $\bG = \ell \bQ$.  
 
Every nontrivial quotient of a Dedekind domain is a principal ideal ring, although not necessarily a domain (see \eg \cite[p.278]{zariski2013commutative}).
This implies the existence of a unique Smith normal form  \cite[Theorem 15.24]{brown1993matrices}. 
That is, there exist $\bU,\bV \in \operatorname{GL}_n(R/I)$ and $d_1,\ldots,d_n \in R$ with $d_{i+1}R + I \subseteq d_{i}R+ I$ for every $i$ such that 
\begin{equation}
    \bG \equiv \bU \bD \bV \bmod I \quad\textnormal{ with }\quad \bD \de \operatorname{diag}(d_1,\ldots,d_n), \label{eq:JollyInk}
\end{equation}
and the reductions of the $d_i$ to $R/ I$ are unique up to multiplication by units.

The main idea in the proof of the following Lemma \ref{lem: UniformSymQ} is to argue that we may replace $\bG$ by the diagonal matrix $\bD$. 
To this end, we rely on an invariance property of uniform symmetric matrices that also plays a key role in a forthcoming work of Lvov and the present author on the satisfaction frequency of sufficient conditions \cite{vanwerde2026satisfaction}.
Recall that the norm of the ideal $I\subseteq R$ is given by $\cN_R(I) = \#R/I$. 
\begin{lemma}\label{lem: UniformSymQ}
    Consider a random symmetric $n\times n$ matrix $\bX\in R^{n\times n}$ whose upper-triangular entries $\{\bX_{i,j}:i\leq j \}$ have independent and uniformly distributed reductions in $R/I$. 
    Then, with $d_1,\ldots,d_n$ as in \eqref{eq:JollyInk}, 
    \begin{equation}
        \bbP\bigl(\bG^{\T} \bX \bG \equiv 0 \bmod I \bigr) = \prod_{i=1}^{n} \prod_{j=i}^{n} \frac{\#\{k\in R/I: kd_id_j \equiv 0 \bmod I\}}{\cN_{R}(I)}. \label{eq:VagueCrow}
    \end{equation}
\end{lemma}
\begin{proof} 
    Note that the uniform law over $R/I$ is characterized by its translation invariance. 
    Consequently, the law of $\bX \bmod I$ is characterized the fact that $\bX$ is symmetric and that $\bX + \bS \bmod I$ has the same law as $\bX$ for every deterministic symmetric $\bS$.  
    Further, for every deterministic symmetric $\bS \in (R/I)^{n\times n}$, it holds with $\bU \in \operatorname{Gl}_n(R/I)$ the invertible matrix from \eqref{eq:JollyInk} that 
    \begin{equation}
        \bU^{\T} \bX \bU + \bS \equiv \bU^{\T}( \bX  + \tilde{\bS}) \bU \bmod I \quad \textnormal{ with }\quad \tilde{\bS} \de (\bU^{-1})^{\T} \bS (\bU^{-1}). 
    \end{equation} 
    Consequently, the translation invariance of the law of $\bX \bmod I$ implies that the law of $\bU^{\T}\bX\bU \bmod I$ is also translation invariant.  
    Hence, $\bU^{\T}\bX\bU \bmod I$ has the same law as $\bX \bmod I$; the uniform law.
    Now, using \eqref{eq:JollyInk},
    \begin{equation}
        \bbP\bigl(\bG^{\T} \bX \bG \equiv 0 \bmod I \bigr)  =  \bbP(\bV^{\T} \bD \bX \bD \bV  \equiv 0 \bmod I).
    \end{equation}
    The invertibility of $\bV \in \operatorname{GL}_n(R/I)$ further implies that $\bV^{\T} \bD \bX \bD \bV  \equiv 0 \bmod I$ if and only if $\bD \bX \bD \equiv 0 \bmod I$. 
    Using the independence and uniformity of the upper-triangular entries now yields \eqref{eq:VagueCrow}.     
\end{proof}
\begin{lemma}\label{lem: CardinalitySimplify}
    It holds for every $d\in R$ that 
    \begin{equation}
        \#\{k \in R /I : kd \equiv 0 \bmod I\} = \cN_{R}(dR + I).\label{eq:ZombieDragon} 
    \end{equation}
\end{lemma}
\begin{proof}
Consider the group morphism
$\phi:R/I \to  R/I$ defined by $\phi(x) = x d$. 
Then, $\#\operatorname{Ker}(\phi)$ is the left-hand side of \eqref{eq:ZombieDragon}.  
Hence, using that $\operatorname{Im}(\phi) \cong (R/I)/\operatorname{Ker}(\phi)$ by the first isomorphism theorem, 
\begin{equation}
    \#\{k \in R/I : kd \equiv 0 \bmod I\} =\#(R/I)/ \# \operatorname{Im}(\phi) = \#(R/I) /{\#((dR + I)/I)}.  
\end{equation}
Note that $\# (R / I) = \#(R/(dR + I))\#((dR +I) /I)$ by the third isomorphism theorem for rings and recall that $\cN_R(dR + I) = \# R/(dR + I)$ to conclude.  
\end{proof}
\begin{theorem}\label{thm: GeneralQ}
    With the assumptions and notation of Lemma \ref{lem: UniformSymQ}, 
    \begin{equation}
        \textstyle \bbP\bigl(\bG^{\T} \bX \bG \equiv 0 \bmod I \bigr) = \prod_{i=1}^{n} \prod_{j=i}^{n} \cN_{R}(d_id_jR + I)/ \cN_{R}(I).\label{eq:StormySax}
    \end{equation}
\end{theorem}
\begin{proof}
    Combine Lemmas \ref{lem: UniformSymQ} and \ref{lem: CardinalitySimplify}.
\end{proof}
\begin{remark}\label{rem: Invariance}
    The invariance-based method also applies to other ensembles. 
    For example, suppose that $\bY$ is a random matrix whose entries are independent \emph{without symmetry constraint} and uniform modulo $I$. 
    Then, for every $\bG_1,\bG_2 \in R^{n\times n}$ with Smith normal forms $d_1^{(1)},\ldots, d_n^{(1)}$ and $d_1^{(2)},\ldots,d_{n}^{(2)}$ over $R/I$,         
    \begin{equation}
        \bbP(\bG_1^{\T}\bY\bG_2 \equiv 0 \bmod I) = \prod_{i=1}^n \prod_{j=1}^n \frac{\cN_R(d_i^{(1)}d_j^{(2)}R + I)}{\cN_R(I)}.\label{eq:SpikySap}
    \end{equation}
    For another, consider a ring automorphism $\tau:R\to R$ with $\tau\circ \tau = \operatorname{Id}_R$. 
    Let $R_{\tau} \de \{r \in R: \tau(r) = r \}$ be the subring of elements that are fixed by $\tau$, and suppose that the ideal $I\subseteq R$ satisfies $\{\tau(i):i\in I \} = I$.
    Denote the matrix found from $\bG \in R^{n\times n}$ by taking the transpose and applying $\tau$ entry-wise by  $
        \bG^\tau \de (\tau(\bG_{j,i}))_{i,j=1}^n.  
    $
    Further, let $\bH$ be a random matrix with $\bH^\tau = \bH$ such that the upper triangular entries are independent with $\bH_{i,j} \bmod I$ uniform on $R/I$ for $i<j$ and $\bH_{i,i} \bmod I\cap R_\tau$ uniform on $R_\tau/(I\cap R_\tau)$. 
    Then, with $d_1,\ldots,d_n\in R$ the Smith normal form of $\bG$ over $R/I$, 
    \begin{equation}
         \bbP\bigl( \bG^{\tau} \bH \bG \equiv 0 \bmod I \bigr) = \Bigl(\prod_{i=1}^n\negsp\frac{\cN_{R_\tau}(\tau(d_i)d_i R_\tau + I\cap R_\tau  )}{\cN_{R_\tau}(I\cap R_{\tau})}\Bigr) \Bigl( \prod_{i=1}^{n-1}\negsp\prod_{j=i+1}^n\negsp \frac{\cN_{R}(\tau(d_i)d_j R + I  )}{\cN_{R}(I
            )} \Bigr).\label{eq:MildMan} 
    \end{equation}
    The proofs of \eqref{eq:SpikySap} and \eqref{eq:MildMan} are minor variations on that of \eqref{eq:StormySax}. 
    (Details can be found in the supplementary material \cite{vanwerde2026exact}.)  
    A concrete example for \eqref{eq:MildMan} would be to take $R=\bbZ[i]$ with $i = \sqrt{-1}$ and let $\tau(z) = \overline{z}$ be complex conjugation. 
    Then, the assumption on $\bH$ means that it is a random Hermitian matrix.
\end{remark}
Thus far, the assumption that $R$ is a Dedekind domain was only used to ensure existence of the Smith normal form over $R/I$, and the orthogonality assumption on $\bQ$ in Theorem \ref{thm: MainExact} has played no role whatsoever. 
The relevance of these assumptions is in the following proofs which establish restrictions on the factors in \eqref{eq:StormySax}.
Recall that the $i$th determinantal ideal $\mathfrak{D}_i\subseteq R$ of $\bG \in R^{n\times n}$ is generated by all $i\times i$ minors.

\begin{lemma}\label{lem: DeterminantalIdealDuality}
    Suppose that $\bG\bG^{\T} = \ell^2 \b1$ for some non-zero $\ell \in R$ and with $\b1$ the identity matrix.
    Then, it holds that $\mathfrak{D}_{n-i} = \ell^{n-2i}\mathfrak{D}_{i}$ for every $i\leq n/2$. 
\end{lemma}
\begin{proof}
    We start by working over the fraction field of $R$. 
    Then, the assumption yields that $\bG^{\T} = \ell^2\bG^{-1}$ over the fraction field.       
    The minors of $\bG^{-1}$ relate to those of the original matrix \cite[Eq.(33)]{gantmakher2000theory}. 
    Specifically, if $\cI,\cJ \subseteq \{1,\ldots,n \}$ are sets of indices with cardinality $i$ and $\cI^{C},\cJ^{C}$ are the complementary sets of indices, then with $\bG_{\cI,\cJ}$ the submatrix having row indices $\cI$ and column indices $\cJ$,
    \begin{equation}
        \det\bigl(\bG_{\cI,\cJ}\bigr) =\bigl(-1\bigr)^{\sum_{i\in \cI}i +\sum_{j\in \cJ}j}    \det\bigl(\bG\bigr)  \det\bigl(\bG^{-1}_{\cI^{C},\cJ^{C}}\bigr).
    \end{equation}
    In particular, since $\bG^{\T} = \ell^2\bG^{-1}$ implies that $\ell^{2 \#\cI^{C}}\det(\bG^{-1}_{\cI^{C},\cJ^{C}})= \det(\bG_{\cJ^C,\cI^C})$,   
    \begin{equation}
         \ell^{2(n- i)} \det(\bG_{\cI,\cJ}) =\pm \det\bigl(\bG\bigr) \det(\bG_{\cJ^{C}, \cI^{C}}).\label{eq:LoudMob}
    \end{equation} 
    Both sides of \eqref{eq:LoudMob} are elements of $R$. 
    It now follows that 
    \begin{equation}
        \ell^{2 (n-i)} \mathfrak{D}_{i} = \det(\bG) \mathfrak{D}_{n-i}. \label{eq:RedNab}
    \end{equation}
    The assumption yields that $\det(\bG)^2 = \ell^{2n}$. 
    Hence, since Dedekind domains enjoy a unique factorization property for ideals \cite[Chapter 5, Theorem 11]{zariski2013commutative}, we find that $\det(\bG) \mathfrak{D}_{n-i} = \ell^n \mathfrak{D}_{n-i}$. 
    Using this in \eqref{eq:RedNab} and again using the unique factorization property to cancel powers of $\ell^n$ yields the claim. 
\end{proof}
Recall that the Smith ideals $\mathfrak{d}_i$ in a Dedekind domain $R$ are defined by $\mathfrak{d}_1 = \mathfrak{D}_1$ and $\mathfrak{d}_{i} \mathfrak{D}_{i-1} = \mathfrak{D}_i$ for every $i>1$ with $\mathfrak{D}_{i-1}\neq 0$. 
Note that $\bG \bG^{\T} = \ell^2 \b1$ with $\ell \neq 0$ implies that $\det(\bG)\neq 0$ and hence $\mathfrak{D}_n \neq 0$. 
The latter implies that $\mathfrak{D}_i \neq 0$ for all $i\leq n$ since $\mathfrak{D}_{i+1} \subseteq \mathfrak{D}_i$ for every $i$ by cofactor expansion. 
\begin{lemma}\label{lem: ProductSmith}
    Adopt the assumptions of Lemma \ref{lem: DeterminantalIdealDuality}.
    Then, it holds that $\mathfrak{d}_{i}\mathfrak{d}_{n-i+1} = \ell^{2}R$ for every $i\leq n/2$. 
\end{lemma}
\begin{proof}
    Using Lemma \ref{lem: DeterminantalIdealDuality} and the calculation rules for fractional ideals in a Dedekind domain that follow from its unique factorization property (\ie the Abelian group structure of the fractional ideals \cite[Chapter 5, Theorem 11]{zariski2013commutative}),
    \begin{equation}
        \mathfrak{d}_{n-i+1} = \mathfrak{D}_{n-i}^{-1}\mathfrak{D}_{n-i+1} = \Bigl(\bigl(\ell^{n-2i} \mathfrak{D}_i\bigr)^{-1}\bigl(\ell^{n-2(i-1)} \mathfrak{D}_{i-1}\bigr)\Bigr) =  \ell^{2 }\mathfrak{D}_{i}^{-1}\mathfrak{D}_{i-1}.   
    \end{equation}
    Using that $\mathfrak{d}_i = \mathfrak{D}_{i-1}^{-1}\mathfrak{D}_{i}$ now yields the desired result.  
\end{proof}
\begin{lemma}\label{lem: SmithIdealSmithNormal}
    Consider a nonzero ideal $I\subseteq R$ as well as an arbitrary matrix $\bG\in R^{n\times n}$ with $\det(\bG)\neq 0$. 
    Let $\mathfrak{d}_1,\ldots,\mathfrak{d}_n$ be the Smith ideals over $R$ and let  $d_1,\ldots,d_n \in R$ be elements associated to the Smith normal form over $R/I$.  
    Then, for every $i\leq n$, 
    \begin{equation}
        d_i R + I = \mathfrak{d}_i + I. \label{eq:ColdQuip}
    \end{equation}
\end{lemma}
\begin{proof} 
    The \emph{cokernel} of $\bG$ is the quotient module $\coker(\bG) \de R^n / \bG R^n$.
    We rely on the fact that Smith ideals are equivalent to the structure of the cokernel when $\det(\bG)\neq 0$.\footnote{While it can be deduced directly from \eqref{eq:JollyInk} that $\mathfrak{D}_i +I = (\prod_{j\leq i}d_j)R + I$ for all $i$, this does not directly imply that $\mathfrak{d}_i + I = d_i R + I$ because $R/I$ may no longer have the unique factorization property for ideals. 
    This is why it is helpful to adopt the canonical perspective of cokernels instead of arguing directly with determinantal ideals.} 
    That is, the Smith ideals satisfy $\mathfrak{d}_{i+1} \subseteq \mathfrak{d}_i$ for every $i$ and they are the unique sequence of length $n$ with this inclusion property such that 
    \begin{equation}
        \coker(\bG) \cong \oplus_{i=1}^n  R/\mathfrak{d}_i. \label{eq:WittyBird}
    \end{equation}
    This relation is classical for principal ideal rings and is often used as the definition of the Smith ideals. 
    The case when $R$ is a Dedekind domain and the $\mathfrak{d}_i$ are defined through determinantal ideals can be deduced from the statement over principal ideal rings; see Appendix \ref{apx: DetailsSmith} for the argument. 

    Denote $\bG_I\in (R/I)^{n\times n}$ for the reduction $\bG \bmod I$.
    Then, a direct computation using \eqref{eq:JollyInk} shows that $\coker(\bG_I) \de (R/I)^n/\bG_I (R/I)^n$ satisfies  
    \begin{equation}
        \coker(\bG_I) \cong \oplus_{i=1}^n  R/(d_i R + I),\label{eq:WobblyYarn}
    \end{equation}
    where we recall from the discussion preceding \eqref{eq:JollyInk} that $d_{i+1}R + I \subseteq d_i R + I$ for every $i$. 
    On the other hand, $\coker(\bG_I) \cong \coker(\bG)\otimes_R (R/I)$ so  \eqref{eq:WittyBird} implies that 
    \begin{equation}
        \coker(\bG_I) \cong \oplus_{i=1}^n R/ (\mathfrak{d}_i + I),  \label{eq:LazyBud}
    \end{equation}
    where it follows from the inclusion $\mathfrak{d}_{i+1}\subseteq \mathfrak{d}_i$ discussed preceding \eqref{eq:WittyBird} (proved in Appendix \ref{apx: DetailsSmith}) that $\mathfrak{d}_{i+1} + I \subseteq \mathfrak{d}_{i} + I$ for every $i$. 
    Combining \eqref{eq:WobblyYarn} and \eqref{eq:LazyBud} with the uniqueness of Smith decompositions of finitely generated modules over the principal ideal ring $R/I$ \cite[Lemma 15.13]{brown1993matrices} now yields \eqref{eq:ColdQuip}. 
\end{proof}

\begin{corollary}\label{cor: MainDedekind}
    Suppose that $\bG\bG^{\T} = \ell^2 \b1$ for some non-zero $\ell \in R$ and consider a random symmetric $n\times n$ matrix $\bX\in R^{n\times n}$ whose upper-triangular entries $\{\bX_{i,j}:i\leq j \}$ have independent and uniformly distributed reductions in $R/\ell^2 R$. 
    Then, 
    \begin{equation}
        \textstyle \bbP\bigl(\bG^{\T} \bX \bG \equiv 0 \bmod  \ell^2R^{n\times n} \bigr) = \prod_{i=1}^{\lfloor n/2 \rfloor} \prod_{j=i}^{n-i} \cN_R(\mathfrak{d}_i \mathfrak{d}_j)/\cN_R(\ell^2 R).\label{eq:GladPizza}
    \end{equation}
\end{corollary}
\begin{proof}
    Recall that $\mathfrak{d}_{i+1} \subseteq \mathfrak{d}_i$ for every $i$. 
    (This was remarked preceding \eqref{eq:WittyBird} and is proved in Appendix \ref{apx: DetailsSmith}.) 
    In particular, Lemma \ref{lem: ProductSmith} hence yields that $\ell^2 R \subseteq \mathfrak{d}_i \mathfrak{d}_{j}$ for every $j \leq n-i+1$ and $\mathfrak{d}_i \mathfrak{d}_{j}\subseteq \ell^2 R$ for every $j\geq n-i+1$. 
    Consequently, using \eqref{eq:ColdQuip} with $I = \ell^2 R$, it holds with $d_1,\ldots,d_n$ as in \eqref{eq:JollyInk} that  
    \begin{equation}
        \cN_R(d_i d_j + \ell^2R) = \cN_R(\mathfrak{d}_i \mathfrak{d}_j + \ell^2R) =
        \begin{cases}
            \cN_R(\ell^2 R)& \textnormal{ if }j\geq n-i+1,\\  
            \cN_R(\mathfrak{d}_i \mathfrak{d}_j) & \textnormal{ if } j\leq n-i. 
        \end{cases}
    \end{equation}
    Hence, \eqref{eq:GladPizza} follows from \eqref{eq:StormySax} since all factors $j> n-i$ equal one.  
\end{proof}
\begin{proof}[Proof of \texorpdfstring{Theorem \ref{thm: GeneralQ}}{Theorem}]
    This is immediate from Corollary \ref{cor: MainDedekind} since it holds that $\bQ^{\T} \bX \bQ \in \sO_K^{n\times n}$ for an orthogonal matrix $\bQ$ of level $\ell$ if and only if $\bG^{\T}\bX \bG \in \ell^2 \sO_K$ for $\bG \de \ell \bQ$. 
    Note that the latter matrix satisfies $\bG^{\T}\bG = \ell^2 \b1$. 
\end{proof}

\section{Proofs for the applications to rational cospectrality}\label{sec: ProofApplication}
\subsection{Proof of \texorpdfstring{Theorem \ref{thm: Switching}}{Theorem}}\label{sec: ProofThm2}
\begin{proposition}\label{prop: Combinatorics2}
    It holds that $\bbO_2(\ell , \bbQ) \neq \emptyset$ if and only if the factorization of $\ell$ into powers of distinct primes is of the form \eqref{eq:KnownToy}. 
    In that case, $\#\bbO_2(\ell , \bbQ) = 2^{r+3}$ with $r$ the number of distinct prime factors of $\ell$. 
\end{proposition}
\begin{proof}
    Consider the primitive decompositions of $\ell^2$ as a sum of two squares: 
    \begin{equation}
        \sR_2^{p}(\ell^2) \de \bigl\{(a,b)\in \bbZ^2 : a^2 + b^2 = \ell^2,\ \operatorname{gcd}(a,b) = 1  \bigr\}.\label{eq:CleverJet}
    \end{equation}
    Then, $2\# \sR_2^p(\ell^2) = \#\bbO_2(\ell , \bbQ)$ since there is a one-to-two correspondence. 
    Indeed, for every $(a,b) \in \sR_2^{p}(\ell^2)$ there are two associated matrices in $\bbO_2(\ell , \bbQ)$ given by  
    \begin{equation}
    \frac{1}{\ell}\begin{pmatrix}
        a &b\\ 
        -b & a
    \end{pmatrix} \qquad  \textnormal{ and }\qquad  \frac{1}{\ell}\begin{pmatrix}
        a &b\\ 
        b &-a
    \end{pmatrix}. 
    \end{equation}
    Specifically, these parametrize the matrices with determinant $1$ and $-1$, respectively.
    The claim now follows from the classical formula for $\#\sR_2^p(\ell^2)$ with $\ell\geq 2$ which yields that this number is zero if $\ell$ has prime factors that are not congruent to $1$ modulo $4$, and equal to $4\times 2^r$ otherwise; see \eg \cite[Eq.(1.6)]{cooper2007number}.  
\end{proof}
\begin{proof}[Proof of \texorpdfstring{Theorem \ref{thm: Switching}}{Theorem}]
    It holds for any $\bQ \in \bbO_2(\ell , \bbQ)$ that the Smith normal form of $\ell\bQ$ is $(d_1, d_2) = (1,\ell^2)$ as is readily checked by considering the determinant and the ideal generated by the entries.
    Hence, decomposing $N_2(\ell)$ as a sum of indicator variables and using the linearity of expectation with Corollary \ref{cor: RationalConjugation}, 
    \begin{equation}
        \bbE[N_2(\ell)] =   \sum_{\bQ \in \bbO_2(\ell , \bbQ)}\bbP (\bQ^\T \bX \bQ  \in \bbZ^{2\times 2})  =  \ell^{-2}\#\bbO_2(\ell , \bbQ). 
    \end{equation}
    The desired result is now immediate from Proposition \ref{prop: Combinatorics2}. 
\end{proof}

\subsection{Proof of \texorpdfstring{Theorem \ref{thm: 3Switching}}{Theorem}}\label{sec: ProofThm3}
We rely on the Euler--Rodrigues parametrization of the special orthogonal group:
\begin{theorem}[Pall {\cite{pall1940rational}}; Cremona {\cite{Galuten01101987}}]\label{thm: EulerRodrigues}
    For every $3\times 3$ rational orthogonal matrix $\bQ\in \bbQ^{3\times 3}$ with $\det(\bQ) = 1$ there exist $a,b,c,d \in \bbZ$ such that   
    \begin{equation}
        \bQ = \frac{1}{a^2 + b^2 + c^2 + d^2}
        \begin{pmatrix}
        a^2 + b^2 -c^2 - d^2& 2(bc - ad) & 2(bd + ac)\\ 
    2(ad + bc) & a^2 - b^2 + c^2 - d^2 & 2(cd - ab)\\ 
    2(bd - ac) & 2(ab + cd) & a^2 -b^2 - c^2 + d^2 \end{pmatrix}\nonumber
    \end{equation}
    Moreover, the vector $v = (a,b,c,d)$ is uniquely determined up to scalar multiplication. 
    In particular, if $\operatorname{gcd}(a,b,c,d)=1$ then the only other vector $v'\in \bbZ^4$ with coprime coordinates that gives the same orthogonal matrix is $v' = - v$.
\end{theorem}
The proof of Theorem \ref{thm: EulerRodrigues} relies on the correspondence between special orthogonal matrices in dimension $3$ and quaternions; see \cite[p.755]{pall1940rational}. 

The denominator $a^2 + b^2 + c^2 + d^2$ is not necessarily the level of the matrix $\bQ$ in Theorem \ref{thm: EulerRodrigues} since it can occur that there are shared divisors with the other entries. 
The following two lemmas serve to overcome this difficulty: 

\begin{lemma}[Pall {\cite{pall1940rational}}]\label{lem: OddLevel}
    It holds that $\bbO_{3}(\ell,\bbQ) = \emptyset$ if $\ell$ is even. 
\end{lemma}
\begin{proof}
    Suppose to the contrary that there exists $\bQ \in \bbO_{3}(\ell,\bbQ)$. 
    The columns of $G \de  \ell \bQ$ are then vectors in $\bbZ^3$ of norm $\ell$. 
    However, the equation $x^2 + y^2 + z^2 \equiv 0 \bmod 4$ only admits the trivial solution $(x,y,z) \equiv (0,0,0) \bmod 4$.   
    (This can be checked using that the only squares in $\bbZ/4\bbZ$ are $0$ and $1$.)
    Thus, if $\ell$ is even, then the entries of $G$ would be divisible by $4$, and hence not coprime to $\ell$ which contradicts the definition of the level. 
\end{proof}
\begin{lemma}[Pall {\cite{pall1940rational}}]\label{lem: gcd4}
    Suppose that $\operatorname{gcd}(a,b,c,d) = 1$. 
    Then, there exists $g\in \{1,2,4 \}$ such that the level $\ell$ of the matrix $\bQ$ in Theorem \ref{thm: EulerRodrigues} is
    $
        \ell = (a^2 + b^2 + c^2 + d^2)/ g. 
    $
\end{lemma}
\begin{proof}
    If $g$ is the greatest common divisor of the denominator and the entries of the matrix in Theorem \ref{thm: EulerRodrigues} then considering the sum of the the diagonal entries and the denominator implies that $g \mid 4 a^2$. 
    Similarly, taking signed combinations yields that $g \mid 4\operatorname{gcd}(b^2,c^2, d^2)$.
    The assumption that $\operatorname{gcd}(a^2,b^2, c^2, d^2 )=1$ now implies that $g\mid 4$, concluding the proof.
\end{proof}
\begin{proposition}\label{prop: Counting3}
    It holds that $\bbO_3(\ell, \bbQ) \neq \emptyset$ if and only if $\ell$ is odd. 
    Moreover, in this case $\#\bbO_3(\ell, \bbQ) = 48 \ell \prod_{p\mid \ell}(1 + 1/p) $.  
\end{proposition}
\begin{proof}
    That $\bbO_3(\ell, \bbQ)$ is empty if $\ell$ is even is immediate from Lemma \ref{lem: OddLevel}. 
    Suppose that $\ell$ is odd.
    Then, considering that exactly half of the matrices in $\bbO_3( \ell, \bbQ)$ have positive determinant as may be seen by exchanging $\bQ$ and $-\bQ$, 
    \begin{equation}
        \#\bbO_3(\ell, \bbQ) = 2\#\{\bQ \in \bbO_3(\ell,\bbQ) : \det(\bQ) = 1 \}. \label{eq:SilentZero}
    \end{equation}  
    Here, using Theorem \ref{thm: EulerRodrigues} with Lemmas \ref{lem: gcd4} and \ref{lem: OddLevel}, 
    \begin{equation}
        \#\{\bQ \in \bbO_3(\ell,\bbQ) : \det(\bQ) = 1 \} =  \frac{1}{2}\Bigl( \#\sR_{4}^p(\ell) + \#\sR_{4}^p(2\ell) + \#\sR_{4}^p(4\ell)\Bigr),\label{eq:ZenToy}
    \end{equation}
    where the factor $1/2$ accounts for the sign ambiguity in Theorem \ref{thm: EulerRodrigues} and $\sR_4(L)$ is the set of primitive decompositions of $L\geq 2$ as a sum of four squares: 
    \begin{equation}
        \sR_4^{p}(L) \de \bigl\{(a,b,c,d)\in \bbZ^2 : a^2 + b^2+ c^2 + d^2 = L,\ \operatorname{gcd}(a,b,c,d) = 1  \bigr\}.
    \end{equation}
    The cardinality of this set is classical (\eg \cite[Eq.(1.7)]{cooper2007number}). 
    In particular, 
    \begin{equation}
        \#\sR_4^p(L) = c_4(L) L \prod_{\textnormal{odd } p\mid L }(1 + 1/p)\label{eq:EasyPop} 
\end{equation}
    where $c_4(L) = 8$ if $L$ is odd, $c_4(L) = 12$ if $L$ is divisible by $2$ but not by $4$, and $c_4(L) = 4$ if $L$ is divisible by $4$ but not by $8$. 
    Combine \eqref{eq:SilentZero}--\eqref{eq:EasyPop} and use that $8+ 2\times 12+4\times 4 = 48$ to conclude. 
\end{proof}
\begin{proof}[Proof of \texorpdfstring{Theorem \ref{thm: 3Switching}}{Theorem}]
    The definition of the level implies that for $\bQ \in \bbO_3(\ell, \bbQ)$ the entries of $\ell \bQ$ do not have a common divisor. 
    Consequently, $\mathfrak{D}_1 = \bbZ$ and hence $d_1 = 1$.   
    Using that the $d_i$ satisfy the division condition $d_1 \mid d_2 \mid d_3$ combined with the constraint $d_i d_{n-i+1} = \ell^2$ from  
    Lemma \ref{lem: ProductSmith} now yields that $(d_1,d_2,d_3) = (1,\ell, \ell^2)$. 
    Consequently, $\prod_{i=1}^{\lfloor n/2 \rfloor }\prod_{j=i}^{n-i} (d_id_j /\ell^2) = 1/\ell^3$ for $n=3$. 
    Using Proposition \ref{prop: Counting3} and the linearity of expectation now concludes the proof. 
\end{proof}
\begin{remark}
    It is interesting to note that primitive decompositions as sums of squares occur in both counting problems for $n=2$ and $n=3$, but in different ways. 
    In \eqref{eq:CleverJet} we counted decompositions of $\ell^2$ in sums of two squares for $n=2$, while the problem with $n=3$ in  \eqref{eq:ZenToy} concerned decompositions of $g\ell$ with $g\in \{1,2,4 \}$ in four squares. 
    Note that $\ell$ is not squared in the latter case. 
\end{remark}

\subsection{Proof of \texorpdfstring{Corollary \ref{cor: Probability}}{Corollary}}\label{sec: ProofCorProbability}
\begin{lemma}\label{lem: DivisibleN}
    For any $n,L\geq 1$ and any $n\times n$ random matrix $\bX$ with entries in $\bbZ$ it holds with $N_n(\ell)$ as in \eqref{eq:QuickPaint} that  
    \begin{equation}
        \bbP\Bigl( \exists \bQ \in \cup_{\ell=2}^L\bbO_n(\ell , \bbQ): \bQ^{\T} \bX \bQ \in \bbZ^{n\times n} \Bigr) \leq \frac{1}{2^n n!} \sum_{\ell=2}^L\bbE[N_n(\ell)].\label{eq:PinkUrn} 
    \end{equation}
\end{lemma}
\begin{proof}
    A \emph{signed permutation} is a matrix of the form $\bS = \bP\bD$ with $\bP$ a permutation matrix and $\bD$ a diagonal matrix with diagonal entries $\pm 1$.  
    Given some $\bQ\in \bbO_n(\ell, \bbQ)$ all the $2^n n!$ matrices in $\bbO_n(\ell, \bbQ)$ of the form $\tilde{\bQ}_S = \bQ\bS$ are distinct due to the orthogonality of the columns. 
    Moreover, $\bQ^{\T} \bX \bQ\in \bbZ^{n\times n}$ if and only if  $\tilde{\bQ}_S^{\T} \bX \tilde{\bQ}_S\in \bbZ^{n\times n}$ for all $\bS$. 
    Hence, $N_n(\ell)$ takes values in $2^n n! \bbZ_{\geq 0}$ and consequently
    \begin{equation}
        \textstyle 2^n n!\bbP(N_n(\ell) \neq 0)  \leq \sum_{m=0}^\infty m \bbP(N_n(\ell) = m) = \bbE[N_n(\ell)].  
    \end{equation}
    The claim \eqref{eq:PinkUrn} now follows from the union bound. 
\end{proof}
\begin{proposition}\label{prop: Upperbound2}
    Suppose that $n=2$ and let $\bX_k$ be as in Corollary \ref{cor: Probability}. 
    Then, 
    \begin{equation}
        \lim_{L\to \infty}\lim_{k\to \infty} \bbP\Bigl( \exists \bQ \in \cup_{\ell=2}^L\bbO_n(\ell , \bbQ): \bQ^{\T} \bX \bQ \in \bbZ^{n\times n} \Bigr) \leq \frac{12 G }{\pi^2} -1,\label{eq:OldBat} 
    \end{equation}
    where  $\pi \approx 3.14...$ and $G$ is \emph{Catalan's constant}, given by  
    \begin{equation}
        G = \sum_{n=0}^\infty \frac{(-1)^n}{(2n+1)^2} =0.9159655941\ldots \label{eq:ShySquid}
    \end{equation}
\end{proposition}
\begin{proof} 
    Note that $\bX_k \bmod \ell$ converges in distribution to the uniform law when $k\to \infty$ and $\ell$ is fixed.
    In particular, the limit over $k$ on the left-hand side of \eqref{eq:OldBat} exists, and the limit over $L$ then also exists since the probability is a non-increasing function of $L$. 
    Further, by \eqref{eq:GhostlyFan} and the fact that that $8 = 2^3$,   
    \begin{align}
        \lim_{k\to \infty}&\sum_{\ell=2}^L \frac{1}{8}\bbE[N_n(\ell)] =   \sum_{\ell=2}^L \ell^{-2}   \prod_{p\mid \ell} \Bigl(2 \bb1\bigl\{p \equiv 1 \bmod 4  \bigr\}  \Bigr) \label{eq:PaleOwl}\\   
        &\leq -1 + \prod_{p \equiv 1 \bmod 4} \Bigl(1 + 2\sum_{j=1}^\infty p^{-2j}\Bigr)  
        = -1 + \prod_{p \equiv 1 \bmod 4}  \frac{1 + p^{-2}}{1-p^{-2}}.  \nonumber
    \end{align}
    Here, the inequality may be verified by expanding the product into a sum, and the final equality is simply the geometric series. 

    The product in \eqref{eq:PaleOwl} can be computed exactly. 
    Indeed, by \cite[\S2.3.1]{finch2003mathematical} there exists a constant $K>0$ with 
    $
        \prod_{p \equiv 1 \bmod 4}( 1 - p^{-2}) = 16K^2 \pi^{-2}$ and $\prod_{p \equiv 1 \bmod 4}(1 + p^{-2}) = 192 K^2 G \pi^{-4}  
    $, 
    where $G$ is Catalan's constant \eqref{eq:ShySquid}. 
    Substitute this in \eqref{eq:PaleOwl} and use Lemma \ref{lem: DivisibleN} to conclude. 
    \end{proof}
    
\begin{proposition}
    Suppose that $n=3$ and let $\bX_k$ be as in Corollary \ref{cor: Probability}. 
    Then, 
    \begin{equation}
        \lim_{L\to \infty}\lim_{k\to \infty} \bbP\Bigl( \exists \bQ \in \cup_{\ell=2}^L\bbO_n(\ell , \bbQ): \bQ^{\T} \bX \bQ \in \bbZ^{n\times n} \Bigr) \leq \frac{105 \zeta(3)}{\pi^4}  - 1, 
    \end{equation}
    where $\zeta(3)$ is \emph{Ap\'ery's constant}: 
    \begin{equation}
        \zeta(3) = \sum_{n=1}^\infty \frac{1}{n^3} =  1.202 056 903 1\ldots
    \end{equation}
\end{proposition}
\begin{proof}
    As in the proof of Proposition \ref{prop: Upperbound2} the limit exists and \eqref{eq:KindRock} yields that 
    \begin{align}
        \lim_{k\to \infty} &\sum_{\ell=2}^L \frac{1}{48} \bbE[N_n(\ell)]  = \frac{48}{48}\sum_{\ell = 2}^L \ell^{-2} \bb1\{2 \nmid \ell\} \prod_{p\mid \ell} \Bigl(1 + p^{-1} \Bigr)\label{eq:ColdPig}\\ 
        &\leq  -1 + \prod_{\textnormal{odd } p}\Bigl( 1 + (1+p^{-1}) \sum_{j=1}^\infty p^{-2j}\Bigr)  =   -1 + \prod_{\textnormal{odd } p}\Bigl(  \frac{1 + p^{-3}}{1-p^{-2}}\Bigr),    \nonumber   
    \end{align}
    where the final step used the geometric series. 
    Further, by the Euler product for the Riemann zeta function, 
    \begin{equation}
        \prod_{\textnormal{odd } p}\Bigl(  \frac{1 + p^{-3}}{1-p^{-2}}\Bigr) = \frac{1-2^{-2}}{1+2^{-3}}\prod_{\textnormal{primes } p}\Bigl(  \frac{1 - p^{-6}}{(1-p^{-2})(1-p^{-3})}\Bigr) = \frac{2}{3}\frac{\zeta(6)^{-1}}{\zeta(2)^{-1}\zeta(3)^{-1} }.  
    \end{equation}
   Here, recall the classical values $\zeta(2)= \pi^2 / 6$ and $\zeta(6) = \pi^6/945$. 
   Using Lemma \ref{lem: DivisibleN} and that $(2\times 945  )/(3 \times 6 )  = 105$ now concludes the proof. 
\end{proof}
\begin{proof}[Proof of \texorpdfstring{Corollary \ref{cor: Probability}}{Corollary}]
    The claim is vacuous if $n=1$ as there are then no rational orthogonal matrices (scalars) with level $\ell \geq 2$. 
    If $n=2$, then Proposition \ref{prop: Upperbound2} yields the claim since $12 G /\pi^2 -1 = 0.11368\ldots <1$. 
    Finally, if $n=3$ then the claim follows from Proposition \ref{prop: Upperbound2} since $105 \zeta(3)/\pi^{4} -1 =0.29573\ldots <1$.  
\end{proof}

\begin{remark}
    That the bound for $n=2$ is smaller than the bound for $n=3$ is due to the restrictive constraints that determine when $\bbP(N_2(\ell) \neq 0) \neq 0$. 
    Indeed, the terms in \eqref{eq:PaleOwl} are greater than the terms in \eqref{eq:ColdPig} for those $\ell$ of the form \eqref{eq:KnownToy} where both are nonzero. 
    The restrictions on the prime factorization are likely to be an exceptional phenomenon that is specific to low dimensionality, so we would expect a decreasing trend in the probabilities as $n$ tends to infinity.   
\end{remark}

\subsection*{Acknowledgements}
I thank Nils van de Berg,  Matthias L\"owe, Nikita Lvov, and Siqi Yang for helpful discussions related to this paper.

Funded by the Deutsche Forschungsgemeinschaft (DFG, German Research Foundation) under Germany's Excellence Strategy EXC 2044/2 –390685587, Mathematics Münster: Dynamics–Geometry–Structure.

\bibliographystyle{abbrv}

\begin{thebibliography}{10}

\bibitem{abiad2022codeterminantal}
A.~Abiad, C.~Alfaro, K.~Heysse, and M.~Vargas.
\newblock Codeterminantal graphs.
\newblock {\em Linear Algebra and its Applications}, 2022.
\newblock \doi{10.1016/j.laa.2022.05.021}.

\bibitem{abiad2025distinguishing}
A.~Abiad, C.~Alfaro, and R.~Villagr{\'a}n.
\newblock Distinguishing graphs by their spectra, {S}mith normal forms and complements.
\newblock {\em Applied Mathematics and Computation}, 2025.
\newblock \doi{10.1016/j.amc.2024.129198}.

\bibitem{abiad2025counting}
A.~Abiad, N.~van~de Berg, and R.~Simoens.
\newblock Counting cospectral graphs obtained via switching.
\newblock {\em Discrete Mathematics}, 2025.
\newblock \doi{10.1016/j.disc.2025.114775}.

\bibitem{brown1993matrices}
W.~Brown.
\newblock {\em Matrices over Commutative Rings}.
\newblock Monographs and Textbooks in Pure and Applied Mathematics. Marcel Dekker, 1992.

\bibitem{collatz1957}
V.~Collatz and U.~Sinogowitz.
\newblock {\em Abhandlungen aus dem Mathematischen Seminar der Universit{\"a}t Hamburg}, chapter Spektren endlicher grafen.
\newblock Springer, 1957.
\newblock \doi{10.1007/BF02941924}.

\bibitem{cooper2007number}
S.~Cooper and M.~Hirschhorn.
\newblock On the number of primitive representations of integers as sums of squares.
\newblock {\em The Ramanujan Journal}, 2007.
\newblock \doi{10.1007/s11139-006-0240-6}.

\bibitem{Galuten01101987}
J.~Cremona.
\newblock Letter to the editor.
\newblock {\em The American Mathematical Monthly}, 1987.
\newblock \doi{10.1080/00029890.1987.12000713}.

\bibitem{elsheikh2015relating}
M.~Elsheikh and M.~Giesbrecht.
\newblock Relating $p$-adic eigenvalues and the local {S}mith normal form.
\newblock {\em Linear Algebra and its Applications}, 2015.
\newblock \doi{10.1016/j.laa.2015.05.001}.

\bibitem{finch2003mathematical}
S.~Finch.
\newblock {\em Mathematical constants}, volume~I.
\newblock Cambridge university press, 2003.
\newblock \href{http://www.cambridge.org/9780521818056}{\nolinkurl{http://www.cambridge.org/9780521818056}}.

\bibitem{gantmakher2000theory}
F.~Gantmacher.
\newblock {\em The theory of matrices}, volume~1.
\newblock {A}merican {M}athematical {S}ociety, 1959.
\newblock \href{https://archive.org/details/gantmacher-the-theory-of-matrices-vol-1-1959}{https://archive.org/details/gantmacher-the-theory-of-matrices-vol-1-1959}.

\bibitem{gordon1992one}
C.~Gordon, D.~Webb, and S.~Wolpert.
\newblock One cannot hear the shape of a drum.
\newblock {\em Bulletin of the American Mathematical Society}, 1992.
\newblock \doi{10.1090/S0273-0979-1992-00289-6}.

\bibitem{gunthard1956zusammenhang}
H.~G{\"u}nthard and H.~Primas.
\newblock {Z}usammenhang von graphentheorie und {M}{O}-theorie von molekeln mit systemen konjugierter bindungen.
\newblock {\em Helvetica Chimica Acta}, 1956.
\newblock \doi{10.1002/hlca.19560390623}.

\bibitem{haemers2016almost}
W.~H. Haemers.
\newblock Are almost all graphs determined by their spectrum.
\newblock {\em Notices of the South African Mathematical Society}, 2016.

\bibitem{johnson1980note}
C.~Johnson and M.~Newman.
\newblock A note on cospectral graphs.
\newblock {\em Journal of Combinatorial Theory, Series B}, 1980.
\newblock \doi{10.1016/0095-8956(80)90058-1}.

\bibitem{kac1966can}
M.~Kac.
\newblock Can one hear the shape of a drum?
\newblock {\em The {A}merican mathematical monthly}, 1966.
\newblock \doi{10.1080/00029890.1966.11970915}.

\bibitem{lang2012algebra}
S.~Lang.
\newblock {\em Algebra}.
\newblock Springer Science \& Business Media, revised third edition, 2002.
\newblock \doi{10.1007/978-1-4613-0041-0}.

\bibitem{vanwerde2026satisfaction}
N.~Lvov and A.~Van~Werde.
\newblock In preparation.

\bibitem{marcus1977number}
D.~Marcus.
\newblock {\em Number fields}.
\newblock Springer, second edition, 2018.
\newblock \doi{10.1007/978-3-319-90233-3}.

\bibitem{milnor1964eigenvalues}
J.~Milnor.
\newblock Eigenvalues of the {L}aplace operator on certain manifolds.
\newblock {\em Proceedings of the National Academy of Sciences}, 1964.
\newblock \doi{10.1073/pnas.51.4.542}.

\bibitem{neukirch1999algebraic}
J.~Neukirch.
\newblock {\em Algebraic number theory}.
\newblock Springer Science \& Business Media, first edition, 1999.
\newblock \doi{10.1007/978-3-662-03983-0}.

\bibitem{newman1991matrices}
M.~Newman and R.~Thompson.
\newblock Matrices over rings of algebraic integers.
\newblock {\em Linear Algebra and its Applications}, 1991.
\newblock \doi{10.1016/0024-3795(91)90284-4}.

\bibitem{pall1940rational}
G.~Pall.
\newblock On the rational automorphs of $x^2_1 + x^2_2+ x^2_3$.
\newblock {\em Annals of Mathematics}, 1940.
\newblock \doi{10.2307/1968855}.

\bibitem{qiu2023smith}
L.~Qiu, W.~Wang, and H.~Zhang.
\newblock Smith normal form and the generalized spectral characterization of graphs.
\newblock {\em {D}iscrete {M}athematics}, 2023.
\newblock \doi{10.1016/j.disc.2022.113177}.

\bibitem{rushanan1995eigenvalues}
J.~Rushanan.
\newblock Eigenvalues and the {S}mith normal form.
\newblock {\em Linear algebra and its applications}, 1995.
\newblock \doi{10.1016/0024-3795(93)00131-I}.

\bibitem{van2003graphs}
E.~van {D}am and W.~{H}aemers.
\newblock {W}hich graphs are determined by their spectrum?
\newblock {\em {L}inear {A}lgebra and its {A}pplications}, 2003.
\newblock \doi{10.1016/S0024-3795(03)00483-X}.

\bibitem{vanwerde2026exact}
A.~Van~Werde.
\newblock Supplementary material for ``{E}xact cospectrality probabilities for uniform random matrices''.
\newblock 2026.
\newblock \href{https://alexandervanwerde.be/wp-content/uploads/2026/01/Supplemental___Exact_cospectrality.pdf}{\nolinkurl{https://alexandervanwerde.be/wp-content/uploads/2026/01/Supplemental___Exact_cospectrality.pdf}}.

\bibitem{VU2014combinatorial}
V.~Vu.
\newblock Combinatorial problems in random matrix theory.
\newblock In {\em Proceedings of the International Congress of Mathematician, Seol 2014}, volume~IV, pages 489--508, 2014.
\newblock \href{https://www.mathunion.org/fileadmin/ICM/Proceedings/ICM2014.4/ICM2014.4.pdf}{\nolinkurl{https://www.mathunion.org/fileadmin/ICM/Proceedings/ICM2014.4/ICM2014.4.pdf}}.

\bibitem{vu2021recent}
V.~Vu.
\newblock Recent progress in combinatorial random matrix theory.
\newblock {\em Probability Surveys}, 2021.
\newblock \doi{10.1214/20-PS346}.

\bibitem{wang2017simple}
W.~Wang.
\newblock A simple arithmetic criterion for graphs being determined by their generalized spectra.
\newblock {\em {J}ournal of {C}ombinatorial {T}heory, {S}eries {B}}, 2017.
\newblock \doi{10.1016/j.jctb.2016.07.004}.

\bibitem{wang2010asymptotic}
W.~Wang and C.-X. Xu.
\newblock On the asymptotic behavior of graphs determined by their generalized spectra.
\newblock {\em Discrete mathematics}, 2010.
\newblock \doi{10.1016/j.disc.2009.07.028}.

\bibitem{wang2024rational}
W.~Wang, J.~Yang, and H.~Zhang.
\newblock Rational orthogonal matrices and isomorphism of graphs.
\newblock {\em Discrete Mathematics}, 2024.
\newblock \doi{10.1016/j.disc.2024.114002}.

\bibitem{wang2016square}
W.~Wang and T.~Yu.
\newblock Square-free discriminants of matrices and the generalized spectral characterizations of graphs.
\newblock {\em arXiv preprint arXiv:1608.01144}, 2016.
\newblock \doi{10.48550/arXiv.1608.01144}.

\bibitem{wang2025almost}
W.~Wang and D.~Zhao.
\newblock Almost all graphs have no cospectral mate with fixed level.
\newblock {\em arXiv preprint arXiv:2509.05781}, 2025.
\newblock \doi{10.48550/arXiv.2509.05781}.

\bibitem{wolpert1979length}
S.~Wolpert.
\newblock The length spectra as moduli for compact {R}iemann surfaces.
\newblock {\em Annals of Mathematics}, 1979.
\newblock \doi{10.2307/1971114}.

\bibitem{zariski2013commutative}
O.~Zariski and P.~Samuel.
\newblock {\em Commutative algebra: Volume I}.
\newblock Springer Science \& Business Media, 1958.
\newblock \href{https://archive.org/details/commutativealgeb0001zari}{\nolinkurl{https://archive.org/details/commutativealgeb0001zari}}.

\end{thebibliography}

\appendix
\section{Details for \texorpdfstring{\eqref{eq:WittyBird}}{Eq.} in the proof of \texorpdfstring{Lemma \ref{lem: SmithIdealSmithNormal}}{Lemma}}\label{apx: DetailsSmith}
    \begin{lemma}
        Let $R$ be a Dedekind domain and consider a square matrix $\bG\in R^{n\times n}$ with $\det(\bG) \neq 0$.
        Then, 
        \begin{enumerate}[leftmargin = 2em]
            \item There exists a sequence of ideals $\mathfrak{a}_1,\ldots,\mathfrak{a}_n \subseteq R$ with the property that $\mathfrak{a}_{i+1}\subseteq\mathfrak{a}_i$ for every $i$ and $\coker(\bG) \cong \oplus_{i=1}^n R/\mathfrak{a}_i$.
            \item It necessarily holds that $\mathfrak{a}_i = \mathfrak{d}_i$ with $\mathfrak{d}_i$ the $i$th Smith ideal.
        \end{enumerate}
    \end{lemma}
    \begin{proof}
        
    It holds for every $v\in R^n$ that $\det(\bG) v =\bG (\bG^{-1}\det(\bG)) v  \equiv 0 \bmod \bG R^n$.
    Thus, $\det(\bG)\coker(\bG) =0$ meaning that $\coker(\bG)$ is not only a $R$-module but also a $R/\det(\bG) R$-module.
    Recall that every nontrivial quotient of a Dedekind domain is a principal ideal ring \cite[p.278]{zariski2013commutative}. 
    The structure theorem for modules over principal ideal rings \cite[Theorem 15.33]{brown1993matrices} hence yields a sequence of ideals $\tilde{\mathfrak{a}}_1,\ldots,\tilde{\mathfrak{a}}_n \subseteq R/\det(\bG)R$ with $\tilde{\mathfrak{a}}_{i+1} \subseteq \tilde{\mathfrak{a}}_i$ for every $i$ such that 
    $
        \coker(\bG) \cong \oplus_{i=1}^n (R/\det(\bG)R)/ \tilde{\mathfrak{a}}_i 
    $ 
    as an $(R/\det(\bG)R)$-module and hence as a $R$-module.
    
    Note that the number of ideals $\tilde{\mathfrak{a}}_i$ here equals the matrix dimension $n$.  
    This is possible by the proof of \cite[Theorem 15.33]{brown1993matrices} which constructs the ideals using the Smith normal form of $\bG$; the latter has exactly $n$ diagonal entries, although some may be units in which case $ (R/\det(\bG)R)/\tilde{\mathfrak{a}}_i =0$. 
    Taking $\mathfrak{a}_i \subseteq R$ to be the unique ideal with $\det(\bG) \in \mathfrak{a}_i$ that reduces to $\tilde{\mathfrak{a}}_i$ modulo $\det(\bG)$ now proves the first property.

    We next prove that $\prod_{k\leq i}\mathfrak{a}_k = \mathfrak{D}_i$ by computing an invariant of the module $\coker(\bG)$ in two ways. 
    The \emph{$i$th Fitting ideal $F_i(\cdot)$} of a module is defined as the $(n-i)$th determinantal ideal of the matrix associated to an arbitrary finite free presentation \cite[Chapter XIX, \S2]{lang2012algebra}. 
    Using the presentation $ R^n \to_{\bG} R^n \to \operatorname{coker}(\bG) \to 0$ hence yields that $
        F_{n-i}\bigl( \coker(\bG)\bigr) \cong \mathfrak{D}_i$. 
    On the other hand, Fitting ideals are an invariant \cite[Chapter XIX, Lemma 2.3]{lang2012algebra}, so  $F_{n-i}\bigl( \coker(\bG)\bigr) = F_{n-i}(\oplus_{k=1}^n  R/\mathfrak{a}_k)$.
    It here holds by \cite[Chapter XIX, Proposition 2.8] {lang2012algebra} that 
    \begin{equation}
        F_{n-i}\Bigl(\oplus_{i=1}^n \frac{R}{\mathfrak{a}_i}\Bigr) = \sum_{\substack{j_1,\ldots,j_n \geq 0\\ \sum_{k=1}^n j_k = n-i}} \prod_{k=1}^n F_{j_k}\Bigl(\frac{R}{\mathfrak{a}_k} \Bigr).\label{eq:TinyInk} 
    \end{equation}
    Further, it holds that $F_0(R/\mathfrak{a}_k) = \mathfrak{a}_k$ by \cite[Chapter XIX, Corollary 2.6]{lang2012algebra} and $F_j(R/\mathfrak{a}_k) = R$ for $j\geq 1$ by \cite[Chapter XIX, Proposition 2.4.(ii)]{lang2012algebra}.    
    Considering the set of indices $k$ with $j_k=0$ in the preceding now yields that
    \begin{equation}
        \mathfrak{D}_i = \sum_{\substack{\cK \subseteq \{1,\ldots,n \}\\ \#\cK \geq i}} \prod_{k\in \cK} F_0\Bigl(\frac{R}{\mathfrak{a}_k} \Bigr) =  \sum_{\substack{\cK \subseteq \{1,\ldots,n \}\\ \#\cK \geq  i}} \prod_{k\in \cK} \mathfrak{a}_k = \prod_{k\leq i} \mathfrak{a}_k. \label{eq:PinkYawn}  
    \end{equation}
    Here, the final equality used that $\mathfrak{a}_{k+1}\subseteq \mathfrak{a}_k$ for all $k$.

    In particular, \eqref{eq:PinkYawn} yields that $\mathfrak{a}_1 = \mathfrak{D}_1$ and that $\mathfrak{a}_i\mathfrak{D}_{i-1} = \mathfrak{D}_i$ for $i\geq 1$. 
    Recall from Definition \ref{def: DeterminantalIdeal} that these properties uniquely define the Smith ideals since $\mathfrak{D}_n \neq 0$ by the assumption that $\det(\bG)\neq 0$ and hence $\mathfrak{D}_i \neq 0$ for all $i$.  
    Hence, $\mathfrak{a}_i = \mathfrak{d}_i$ for all $i$, as desired. 
    \end{proof}
\end{document}